\documentclass[a4paper,12pt,reqno]{amsart}

\usepackage{amsmath}
\usepackage{amscd}
\usepackage{amssymb}
\usepackage{mathrsfs}
\usepackage[left=2.5cm,right=2.5cm,bottom=3cm,top=3cm]{geometry}
\newtheorem{thm}{Theorem}[section]
\newtheorem{cor}[thm]{Corollary}

\newtheorem{ques}[thm]{Question}
\newtheorem{prop}[thm]{Proposition}
\theoremstyle{definition}
\newtheorem{defn}[thm]{Definition}
\newtheorem{notation}[thm]{Notation}
\newtheorem{rem}[thm]{Remark}

\numberwithin{equation}{section}

\begin{document}

\title{Integral inequalities for holomorphic maps and applications}

\author{Yashan Zhang}

\address{School of Mathematics and Hunan Province Key Lab of Intelligent Information Processing and 
Applied Mathematics, Hunan University, Changsha 410082, China}
\email{yashanzh@hnu.edu.cn}

\thanks{The author is partially supported by Fundamental Research Funds for the Central Universities (No. 531118010468) and National Natural Science Foundation of China (No. 12001179)}

\begin{abstract}
We derive some integral inequalities for holomorphic maps between complex manifolds. As applications, some rigidity and degeneracy theorems for holomorphic maps without assuming any pointwise curvature signs for both the domain and target manifolds are proved, in which key roles are played by total integration of the function of the first eigenvalue of second Ricci curvature and an almost nonpositivity notion for holomorphic sectional curvature introduced in our previous work. We also apply these integral inequalities to discuss the infinite-time singularity type of the K\"ahler-Ricci flow. The equality case is characterized for some special settings.
\end{abstract}

\maketitle

\section{Introduction}
\subsection{Background}A general principle in complex geometry states that \emph{negative curvature restricts behaviors of holomorphic maps between complex manifolds}, see e.g. \cite[page 15]{KL}. The most classic result along this line should be Schwarz-Pick-Ahlfors Lemma: a non-constant holomorphic map from the unit disc (equipped with the Poincar\'e metric) to a smooth Riemann surface with negative curvature decreases distances (up to multiplying a constant factor depending only on the bounds of curvatures). Important generalizations of Schwarz-Pick-Ahlfors Lemma to higher dimensions were developed, including Chern \cite{Chern}, Lu \cite{L}, Yau \cite{Y78}, Royden \cite{R}, etc.. Here, we in particular recall Yau's general Schwarz Lemma \cite{Y78} that a holomorphic map from a complete K\"ahler manifold of \emph{Ricci curvature} bounded from below to a Hermitian manifold of \emph{holomorphic bisectional curvature} bounded from above by a negative constant decreases distances (up to multiplying a constant factor depending only on the bounds of curvatures). Royden \cite{R} proved that similar result holds if the target space is a K\"ahler manifold of \emph{holomorphic sectional curvature} bounded from above by a negative constant. In particular, their results imply a fundamental rigidity theorem that \emph{a holomorphic map from a compact K\"ahler manifold of positive Ricci curvature to a Hermitian (resp. K\"ahler) manifold of nonpositive holomorphic bisectional (resp. sectional) curvature must be constant}. Excellent expositions on differential geometric developments of Schwarz-type Lemma can be found in \cite{KL}. More recently, there are significant progresses on this topic, which relaxed either the curvature assumptions or K\"ahlerian condition, see \cite{Ni,To,Y1,Y2,YZ} and references therein for more details.

\subsection{Motivations}\label{moti}
Let's focus on the rigidity theorems. The general philosophy behind rigidity theorems for holomorphic maps is that a holomorphic map from a positively curved space to a negatively curved space should be constant. Our study here is mainly motivated by the following natural question: can we make this philosophy more effective? To be more precise, let's look at, for example,  the aforementioned fundamental rigidity theorem of Yau and Royden: \emph{a holomorphic map from a compact K\"ahler manifold of positive (resp. nonnegative) Ricci curvature to a compact K\"ahler manifold of nonpositive (resp. negative) holomorphic sectional curvature must be constant}. For convenience, here we have assumed that both the domain and target manifolds are compact K\"ahler. Consequently, given a non-constant holomorphic map $f:(X,\omega)\to(Y,\eta)$ between two compact K\"ahler manifolds, 
\begin{itemize}
\item [(I)] if the Ricci curvature of $\omega$ is positive, then the supremum of holomorphic sectional curvature, denoted by $\sup_YH^\eta$ (see Notation \ref{hc}), must be positive;
\item [(II)] if the holomorphic sectional curvature of $\eta$ is negative, then the infimum of Ricci curvature of $\omega$, denoted by $\inf_X\lambda_\omega$ (see Notation \ref{fe}), must be negative.
\end{itemize}
Then, regarding the effectiveness, we may naturally ask (in the above setting):
\begin{itemize}
\item[(A)] in the above case (I), can we have an effective positive lower bound for $\sup_YH^\eta$?
\item[(B)] in the above case (II), can we have an effective negative upper bound for $\inf_X\lambda_\omega$?
\end{itemize}
The effectiveness results, if can be obtained, will play crucial roles in several problems, including 
\begin{itemize}
\item weakening/sharpening the curvature conditions in the rigidity theorems for holomorphic maps;
\item dealing with a family of K\"ahler/Hermitian metrics on either the domain or target manifolds, which naturally arises in the study of complex geometric flows including the K\"ahler-Ricci flow and the Chern-Ricci flow.  
\end{itemize}

\par Firstly, we observe that the case (B) readily follows from Yau's general Schwarz Lemma. Precisely, in the above case (II) there holds ($dim_{\mathbb C}Y=m$)
\begin{align}\label{sl}
f^*\eta\le\frac{\inf_X\lambda_\omega}{\frac{m+1}{2m}\sup_YH^\eta}\cdot\omega
\end{align}
on $X$ and hence
\begin{align}\label{int_sl}
\inf_X\lambda_\omega\le\frac{\frac{m+1}{2m}(\sup_YH^\eta)\int_Xf^*\eta\wedge\omega^{n-1}}{\int_X\omega^n},
\end{align}
giving an effective negative upper bound for $\inf_X\lambda_\omega$. One may interpret \eqref{int_sl} as an integral version of \eqref{sl}.\\

\par In the following, let's focus on the case (A). Under the curvature conditions in case (I), we don't have the ``reverse" Schwarz Lemma, (say, $f^*\eta\ge\frac{\inf_X\lambda_\omega}{\frac{m+1}{2m}\sup_YH^\eta}\cdot\omega$), leaving the (integral) inequality \eqref{int_sl} unclear. Noting that \eqref{int_sl} should be much weaker than the (``reverse") Schwarz Lemma, to solve the above case (A) we are naturally leaded to prove the inequality \eqref{int_sl} directly. In this paper, we shall prove a general integral inequality for non-constant holomorphic maps in a general setting without assuming any curvature conditions (Theorem \ref{main_const}), which in particular implies the desired inequality \eqref{int_sl} in its setting (up to the constant factor $\frac{m+1}{2m}$), and hence solves the above case (A) completely. We should mention that for the special case $(X,\omega)=(\mathbb{CP}^1,\omega_{FS})$, an inequality similar to \eqref{int_sl} was previously proved in Tosatti-Y.G.Zhang \cite[Section 4, Remark 4.1]{ToZyg} by a different method (see subsection \ref{subsect_cp1} for some more discussions).
 
\subsection{Main results: integral inequalities for holomorphic maps}
In this paper, we shall present several integral inequalities for non-constant or non-degenerate holomorphic maps between two complex manifolds without assuming any curvature condition, in which we assume the domain manifold is compact. Precisely, we shall prove the followings:

\begin{thm}\label{main_const}
Let $(X,\omega)$ and $(Y,\eta)$ be two Hermitian manifolds and $(X,\omega)$ be $n$-dimensional and compact. Then there exists a smooth real function $\phi$ on $X$ such that for any non-constant holomorphic map $f:X\to Y$ there holds
\begin{equation}\label{main.ineq}
\int_X\lambda_{\omega}e^{(n-1)\phi}\omega^n\le n\int_Xf^*\kappa_\eta \cdot e^{(n-1)\phi}f^*\eta\wedge\omega^{n-1},
\end{equation}
where \\
(1) $\lambda_{\omega}$ is the function of the first eigenvalue of the second Ricci curvature $Ric^{(2)}(\omega)$ of $\omega$ with respect to $\omega$, and \\
(2) $\kappa_\eta$ is a continuous real function on $Y$ such that for $y\in Y$, $\kappa_\eta(y)$ is the maximal value of holomorphic bisectional curvature of $\eta$ at $y$ when $\eta$ is not K\"ahlerian, and is the ``modified" maximal value of holomorphic sectional curvature of $\eta$ at $y$ when $\eta$ is K\"ahlerian (see Notation \ref{hc} for precise definition).
\par If furthermore $\omega$ is Gauduchon, the $\phi$ in \eqref{main.ineq} can be chosen to be any constant function.
\end{thm}

\begin{thm}\label{main_de}
Let $(X,\omega)$ and $(Y,\eta)$ be two $n$-dimensional Hermitian manifolds and $(X,\omega)$ compact. Then there exists a smooth real function $\psi$ on $X$ such that for any non-degenerate holomorphic map $f:X\to Y$ there holds
\begin{equation}\label{ineq_de}
\int_XR_{\omega}e^{(n-1)\psi}\omega^n\le n\int_Xe^{(n-1)\psi}f^*(Ric(\eta))\wedge\omega^{n-1},
\end{equation}
where $R_{\omega}$ is the Chern scalar curvature of $\omega$ and $Ric(\eta)$ is the Chern Ricci curvature of $\eta$.
\par If furthermore $\omega$ is Gauduchon, the $\psi$ in \eqref{ineq_de} can be chosen to be any constant function.
\end{thm}

\begin{rem}\label{rem_1}
These integral inequalities may be regarded as \emph{effective} obstructions for a holomorphic map being constant or totally degenerate, and they make the aforementioned philosophy in Subsection \ref{moti} more effective. In particular, inequality \eqref{main.ineq} implies a complete answer to the case (A) in Subsection \ref{moti}, i.e. \eqref{int_sl} holds in its setting (also see Subsection \ref{final_rem} for more general results).
\end{rem}

\subsection{The  case $X=\mathbb{CP}^n$ in Theorem \ref{main_const}}\label{subsect_cp1} Let's take a closer look at Theorem \ref{main_const} with the domain space $(\mathbb{CP}^n,\omega_{FS})$. 

Firstly, we recall that for the one-dimensional case, the following result was proved in \cite[Section 4, page 2938-2940]{ToZyg}.
\begin{prop}\cite{ToZyg}
Let $(Y,\eta)$ be a compact K\"ahler manifold and $f:\mathbb{CP}^1\to Y$ a non-constant holomorphic map. Then there holds
\begin{equation}\label{tz_1}
(\sup_YH_\eta)\cdot\int_{\mathbb{CP}^1}f^*\eta\ge\frac\pi8.
\end{equation}
If furthermore $f:\mathbb{CP}^1\to Y$ is a holomorphic embedding, then
\begin{equation}\label{tz_2}
(\sup_YH_\eta)\cdot\int_{\mathbb{CP}^1}f^*\eta\ge2\pi\cdot\chi(\mathbb{CP}^1)=4\pi.
\end{equation}
\end{prop}

By arguments in \cite{ToZyg}, \eqref{tz_2}, which serves as an intuition for \eqref{tz_1}, was a consequence of Gauss-Bonnet Theorem, and \eqref{tz_1} was proved by using Schwarz Lemma calculation and $\epsilon$-regularity arguments.
In particular, the lower bound $4\pi$ in the latter inequality \eqref{tz_2} is automatically of geometric interpretation (i.e. Euler characteristic of $\mathbb{CP}^1$), while the lower bound $\frac\pi8$ in the inequality \eqref{tz_1}, coming from an analytic argument, seems not of clear geometric nature. It is then very natural to ask: \emph{can we have a ``geometric" positive lower bound in \eqref{tz_1}, and can we improve the lower bound in \eqref{tz_1} to $4\pi$? Furthermore, does there exist higher dimensional analog of \eqref{tz_1}?}\\

Using a special case in Theorem \ref{main_const}, i.e. setting $(X,\omega)=(\mathbb{CP}^1,\omega_{FS})$, we answer these questions as follows.

\begin{prop}\label{prop_cp1}
Let $(Y,\eta)$ be a compact K\"ahler manifold and $f:\mathbb{CP}^1\to Y$ a non-constant holomorphic map. Then there holds
\begin{equation}\label{tz_3}
(\sup_YH_\eta)\cdot\int_{\mathbb{CP}^1}f^*\eta\ge2\pi\cdot\chi(\mathbb{CP}^1)=4\pi.
\end{equation}
\end{prop}
Indeed, for $(\mathbb{CP}^1,\omega_{FS})$ with normalization $\int_{\mathbb{CP}^1}\omega_{FS}=1$,  we know $\lambda_{\omega_{FS}}\equiv4\pi$ is the constant Gaussian curvature $K_{\omega_{FS}}$ and hence by Gauss-Bonnet theorem, 
$$\int_{\mathbb{CP}^1}\lambda_{\omega_{FS}}\omega_{FS}=\int_{\mathbb{CP}^1}K_{\omega_{FS}}\omega_{FS}=2\pi\chi(\mathbb{CP}^1).$$ Then Proposition \ref{prop_cp1} follows immediately from Theorem \ref{main_const}. 

Observe that the lower bound $4\pi$ is somehow \emph{optimal}, as can be seen by choosing $(Y,\eta)=(\mathbb{CP}^1,\omega_{FS})$ and $f=Id_{\mathbb{CP}^1}$.

The above arguments can be easily generalized to the cace with domain $(\mathbb{CP}^n,\omega_{FS})$, where $\int_{\mathbb{CP}^n}\omega_{FS}^n=Vol(\mathcal O_{\mathbb{CP}^n}(1))=1$ and $\lambda_{\omega_{FS}}\equiv2\pi(n+1)$. Namely, we have

\begin{prop}\label{prop_cpn}
Let $(Y,\eta)$ be a compact K\"ahler manifold and $f:\mathbb{CP}^n\to Y$ a non-constant holomorphic map. Then there holds
\begin{equation}
(\sup_YH_\eta)\cdot\int_{\mathbb{CP}^n}f^*\eta\wedge\omega_{FS}^{n-1}\ge\frac{2\pi(n+1)}{n}\cdot Vol(\mathcal O_{\mathbb{CP}^n}(1))=\frac{2\pi(n+1)}{n}\nonumber.
\end{equation}
\end{prop}

\subsection{Outline of applications} As described in Subsection \ref{moti}, the effectiveness results Theorems \ref{main_const} and \ref{main_de} may have applications in several problems. Let's outline some of these applications.
\par A classical differential geometric approach in proving rigidity theorem for holomorphic map (i.e. proving constancy of a holomorphic map) makes use of Chern-Lu formula, pointwise curvature signs and the maximum principle arguments. Roughly speaking, this approach makes use of pointwise curvature signs to destroy (pointwise) Schwarz-type Lemma and hence gets the constancy of holomorphic maps. 
\par As applications of Theorem \ref{main_const}, we shall prove rigidity theorems under weaker curvature conditions; in particular, the curvatures are not necessarily pointwise signed. In fact, given Theorem \ref{main_const}, to prove constancy of holomorphic map, it suffices to destroy the integral inequality \eqref{main.ineq}, which can be achieved by just assuming, for example, suitable signs for curvature in certain \emph{integral} sense or ``almost" sense. Therefore, we obtain new rigidity theorems for holomorphic maps without assuming any pointwise curvature signs for both the domain and target manifolds. 
\par Similarly, Theorem \ref{main_de} can be applied to prove degeneracy theorems for holomorphic maps without assuming any pointwise curvature signs for both the domain and target manifolds. 

\par Moreover, our Theorem \ref{main_const} implies a criterion for type \textrm{IIb} singularities of the K\"ahler-Ricci flow, generalizing a result of Tosatti-Y.G. Zhang \cite[Proposition 1.4]{ToZyg}. 

\par Also, the equality case in Theorem \ref{main_const} is characterized in some special settings.

\par The details of these applications will be given in Section \ref{app}.

\subsection{Organization} In the next section, we will introduce some necessary notations and results in complex geometry. Then we prove our main results Theorems \ref{main_const} and \ref{main_de} in Section \ref{pf}. Finally, in Section \ref{app}, we provide several applications, including several rigidity and degeneracy theorems for holomorphic maps, a criterion for type \textrm{IIb} singularities of the K\"ahler-Ricci flow and characterization of equality case in some settings.

\section{Preliminaries} 
\subsection{Curvatures in complex geometry} Let $(X,\omega)$ be a Hermitian manifold of dimension $dim_{\mathbb C}X=n$, where $\omega=\omega_g$ is the K\"ahler form of a Hermitian metric $g$. In a local holomorphic chart $(z^1,...,z^n)$, we write
$$\omega=\sqrt{-1}g_{i\bar j}dz^i\wedge d\bar z^j.$$
Recall the curvature tensor $R^\omega=\{R^\omega_{i\bar jk\bar l}\}$ of the Chern connection is given by 
$$R^\omega_{i\bar jk\bar l}=-\frac{\partial^2g_{k\bar l}}{\partial z^i\partial\bar z^j}+g^{\bar qp}\frac{\partial g_{k\bar q}}{\partial z^i}\frac{\partial g_{p\bar l}}{\partial \bar z^j}.$$
Then the Chern Ricci curvature $Ric(\omega)=\sqrt{-1}R_{i\bar j}dz^i\wedge d\bar z^j$ and Chern scalar curvature $R_\omega$ are given by 
$$R_{i\bar j}=g^{\bar lk}R^\omega_{i\bar jk\bar l}$$
and 
$$R_\omega=g^{\bar ji}R_{i\bar j}$$
respectively.
Note that $R_{i\bar j}=-\frac{\partial^2\log \det(g_{p\bar q})}{\partial z^i\partial\bar z^j}$. \\
Also recall that the second Ricci curvature $Ric^{(2)}(\omega)$ is given by 
$$R_{k\bar l}^{(2)}=g^{\bar ji}R^\omega_{i\bar jk\bar l}.$$

\begin{notation}[Function of the first eigenvalue of the second Ricci curvature]\label{fe}
Let $(X,\omega)$ be a Hermitian manifold of $dim_{\mathbb C}X=n$. Define $\lambda_{\omega}$ to be the function of the first eigenvalue of the second Ricci curvature $Ric^{(2)}(\omega)$ of $\omega$  with respect to $\omega$, which is a continuous function on $X$ satisfying
$$Ric^{(2)}(\omega)\ge\lambda_\omega\omega.$$
For convenience, we call the integration $\int_X\lambda_{\omega}\omega^n$ the \emph{total first eigenvalue} of second Ricci curvature of $\omega$. 
\end{notation}

Given $x\in X$ and $W,V\in T^{1,0}_xX\setminus\{0\}$, the holomorphic bisectional curvature of $\omega$ at $x$ determined by $W,V$ is 
$$BK^\omega_{x}(W,V):=\frac{R^\omega(W,\overline W,V,\overline V)}{|W|^2_\omega|V|^2_\omega},$$ 
and the holomorphic sectional curvature of $\omega$ at $x$ in the direction $W$ is
$$H^\omega_x(W):=\frac{R(W,\overline W,W,\overline W)}{|W|^4_\omega}.$$

\begin{notation}[Supremum of holomorphic curvature]\label{hc}
Let $(X,\omega)$ be a Hermition manifold of $dim_{\mathbb C}X=n$ and $x\in X$. We set 
$$BK^\omega_x:=\sup \{BK^\omega_x(W,V)|W,V\in T^{1,0}_xX\setminus\{0\}\}$$
and 
$$H^\omega_x:=\sup \{H^\omega_x(W)|W\in T^{1,0}_xX\setminus\{0\}\}.$$ 
Then we define a continuous real function $\kappa_\omega$ for $(X,\omega)$ as follows.
\begin{itemize}
\item[(1)] when $\omega$ is not K\"ahlerian, we define $\kappa_\omega(x):=BK^\omega_x$;
\item[(2)] when $\omega$ is K\"ahlerian, we define $\kappa_\omega(x):=\rho(H^\omega_x)\cdot H^\omega_x$, where $\rho:\mathbb R\to\{\frac{n+1}{2n},1\}$ is a function with $\rho(s)=\frac{n+1}{2n}$ for $s\le0$ and $\rho(s)=1$ for $s>0$. 
\end{itemize}
\end{notation}

\begin{defn}[Compact K\"ahler manifold of almost nonpositive holomorphic sectional curvature \cite{Zys18}]\label{defn1}
Let $(X,\hat\omega)$ be a compact K\"ahler manifold.
\begin{itemize}
\item[(1)] Let $\alpha$ be a K\"ahler class on $X$. We define a number  $\mu_\alpha$ for $\alpha$ in the following way: \\
\centerline{$\mu_\alpha:=\inf\{\sup_XH^\omega|\omega$ is a K\"ahler metric in $\alpha\}$,}
where $H^\omega$ is the continuous function on $X$ with $H^\omega(x)=H^\omega_x$.
\item[(2)] We say $X$ is of \textbf{almost nonpositive holomorphic sectional curvature} if for any number $\epsilon>0$, there exists a K\"ahler class $\alpha_\epsilon$ on $X$ such that $\mu_{\alpha_\epsilon}\alpha_\epsilon<\epsilon[\hat\omega]$.
\end{itemize}
\end{defn}
One may find more motivations and discussions about the above almost nonpositivity notion for holomorphic sectional curvature in \cite{Zys18}. Here we mention that it is not a pointwise notion, but a notion at the level of $(1,1)$-classes, and the number $\mu_\alpha$, up to multiplying a constant factor depending only on dimension of manifold, turns out to be an upper bound for the nef threshold of $\alpha$ \cite[Proposition 1.9]{Zys18}.

\subsection{Royden's trick} Let $f:(X,\omega)\to (Y,\eta)$ be a holomorphic map between two Hermitian manifolds and $dim_{\mathbb C}X=n$ and $dim_{\mathbb C}Y=m$. Given $x\in X$ and holomorphic charts $(z^1,...,z^n)$ on $X$ centered at $x$ and $(w^1,...,w^m)$ on $Y$ centered at $f(x)$. Write $f=(f^1,...,f^m)$ and $f^\alpha_i:=\frac{\partial f^\alpha}{\partial z^i}$, $1\le i\le n, 1\le\alpha\le m$, in these local charts. Assuming that $\eta$ is K\"ahler, Royden \cite[page 552]{R} proved that at $x$,
$$R^\eta_{\alpha\bar\beta\gamma\bar\delta}\left(g^{\bar ji}f^{\alpha}_i\overline{f^\beta_j}\right)\left(g^{\bar lk}f^{\gamma}_k\overline{f^\delta_l}\right)\le f^*\kappa_\eta \cdot (tr_\omega f^*\eta)^2,$$
where $\kappa_\eta$ is the function defined in Notation \ref{hc} (2).

\subsection{Gauduchon metrics}\label{subsect_g} Let $X$ be a compact complex manifold of $\dim_{\mathbb C}X=n$. A Hermitian metric $\omega$ on $X$ is called Gauduchon if
$$\partial\bar\partial(\omega^{n-1})=0.$$
Obviously, for a Gauduchon metric $\omega$ and a smooth function $u$ on $X$ we have
$$\int_X(\Delta_\omega u)\omega^n=0,$$
where $\Delta_\omega u$ is the complex Laplacian defined by $\Delta_\omega u=tr_\omega(\sqrt{-1}\partial\bar\partial u)$. 
A classical result of Gauduchon \cite{G} states that, for every Hermitian metric $\omega$, there is a $\phi\in C^\infty(X,\mathbb R)$ (unique up to scaling, when $n\ge2$) such that $e^\phi\omega$ is Gauduchon.

\subsection{Non-degenerate holomorphic maps} 
Let $f:X\to Y$ be a holomorphic map between two complex manifolds of the same dimension. If there exists some point $x\in X$ such that the Jacobian $J(f)$ of $f$ satisfies $|\det J(f)|(x)>0$, then we say $f$ is non-degenerate; otherwise, we say $f$ is totally degenerate.

\section{Proofs of integral inequalities}\label{pf}
In this section we prove Theorems \ref{main_const} and \ref{main_de}. Our proofs do \emph{not} involve any maximum principle arguments, since the curvatures of both the domain and target spaces may not be signed in pointwise sense. Instead, we will make use of a perturbation method involving Lebegue's Dominated Convergence Theorem.

\begin{proof}[Proof of Theorem \ref{main_const}]
We only consider the case that $\eta$ is a K\"ahler metric on $Y$, as the other case can be proved similarly. 
\par Let $\epsilon$ be an arbitrary positive number, $\Lambda=tr_\omega f^*\eta=|\partial f|^2=|df|^2$. Given an arbitrary $x\in X$, and holomorphic charts $(z^1,...,z^n)$ on $X$ centered at $x$ and $(w^1,...,w^m)$ on $Y$ centered at $f(x)$. Write $\omega=\sqrt{-1}g_{i\bar j}dz^i\wedge d\bar z^j$, $\eta=\sqrt{-1}\eta_{\alpha\bar\beta}dw^\alpha\wedge d\bar w^\beta$, $Ric^{(2)}(\omega)=\sqrt{-1}R^{(2)}_{i\bar j}dz^i\wedge d\bar z^j$, $f=(f^1,...,f^m)$ and $f^\alpha_i:=\frac{\partial f^\alpha}{\partial z^i}$, $1\le i\le n, 1\le\alpha\le m$, in these local charts. By direct computations and Chern-Lu formula \cite{Chern,L,Y} (also see \cite[Lemma 4.1]{YZ} for clearer and simpler discussions on Chern-Lu formula) we have
\begin{align}\label{cl}
\Delta_{\omega}\log (\Lambda+\epsilon)&=\frac{\Delta_\omega\Lambda}{\Lambda+\epsilon}-\frac{|\partial\Lambda|^2}{(\Lambda+\epsilon)^2}\nonumber\\
&=\frac{R^{(2)}_{k\bar l}g^{\bar qk}g^{\bar lp}\eta_{\alpha\bar\beta}f^\alpha_p\overline{f^\beta_q}-R^\eta_{\alpha\bar\beta\gamma\bar\delta}\left(g^{\bar ji}f^{\alpha}_i\overline{f^\beta_j}\right)\left(g^{\bar lk}f^{\gamma}_k\overline{f^\delta_l}\right)+|\nabla df|^2}{\Lambda+\epsilon}-\frac{|\partial\Lambda|^2}{(\Lambda+\epsilon)^2}\nonumber\\
\end{align}
By the definition of $\lambda_\omega$ in Notation \ref{fe} and Royden's trick, we get
\begin{align}\label{rigid_ineq}
\Delta_{\omega}\log (\Lambda+\epsilon)&\ge\lambda_\omega\cdot\frac{\Lambda}{\Lambda+\epsilon}-f^*\kappa_\eta\cdot\Lambda\cdot\frac{\Lambda}{\Lambda+\epsilon}+\frac{|\nabla df|^2}{\Lambda+\epsilon}-\frac{|\partial\Lambda|^2}{(\Lambda+\epsilon)^2}
\end{align}

For the last two terms in \eqref{rigid_ineq}, 
\begin{align}\label{est_ineq}
\frac{|\nabla df|^2}{\Lambda+\epsilon}-\frac{|\partial\Lambda|^2}{(\Lambda+\epsilon)^2}&=\frac{|\nabla df|^2}{\Lambda+\epsilon}-\frac{|\partial|df|^2|^2}{(\Lambda+\epsilon)^2}\nonumber\\
&=\frac{|\nabla df|^2}{\Lambda+\epsilon}-\frac{4|\partial|df||^2}{(\Lambda+\epsilon)}\frac{\Lambda}{\Lambda+\epsilon}\nonumber\\
&=\frac{|\nabla df|^2}{\Lambda+\epsilon}-\frac{|d|df||^2}{(\Lambda+\epsilon)}\frac{\Lambda}{\Lambda+\epsilon}\nonumber\\
&=\frac{\epsilon|\nabla df|^2}{(\Lambda+\epsilon)^2}+\frac{|\nabla df|^2-|d|df||^2}{(\Lambda+\epsilon)}\frac{\Lambda}{\Lambda+\epsilon}
\end{align}

Set $V:=\{x\in X|tr_{\omega}f^*\eta=0$ at $x\}=\{x\in X|\partial f=0$ at $x\}$, which, when $f$ is non-constant, is a proper subvariety (may be empty) of $Y$. Now, as in \cite[Lemma 4.2]{YZ}, we apply Kato inequality in \eqref{est_ineq} to conclude that, outside $V$, 
\begin{align}
\frac{|\nabla df|^2}{\Lambda+\epsilon}-\frac{|\partial\Lambda|^2}{(\Lambda+\epsilon)^2}\ge\frac{\epsilon|\nabla df|^2}{(\Lambda+\epsilon)^2}\ge0\nonumber,
\end{align}
putting which into \eqref{rigid_ineq} gives, on $X\setminus V$, 
\begin{align}\label{rigid_ineq_1}
\Delta_{\omega}\log (\Lambda+\epsilon)&\ge\lambda_\omega\cdot\frac{\Lambda}{\Lambda+\epsilon}-f^*\kappa_\eta\cdot\Lambda\cdot\frac{\Lambda}{\Lambda+\epsilon}
\end{align}

Note that both sides of \eqref{rigid_ineq_1} are continuous functions on $X$ and $V$ is a proper subvariety, therefore, by continuity \eqref{rigid_ineq_1} in fact holds on the whole $X$.

Next, we fix a $\phi\in C^\infty(X,\mathbb R)$ such that $e^\phi\omega$ is a Gauduchon metric on $X$ (see subsection \ref{subsect_g}).

Integrating \eqref{rigid_ineq_1} with respect to $e^{(n-1)\phi}\omega^n$ over $X$ gives
\begin{align}\label{rigid_ineq_2}
\int_X\frac{\Lambda}{\Lambda+\epsilon}\left(\lambda_{\omega}-f^*\kappa_\eta \Lambda\right)e^{(n-1)\phi}\omega^n\le\int_X\left(\Delta_{\omega}\log (\Lambda+\epsilon)\right)e^{(n-1)\phi}\omega^n.
\end{align}
Since $e^\phi\omega$ is Gauduchon and $\log (\Lambda+\epsilon)$ is a smooth real function on $X$, we have
$$\int_X\left(\Delta_{\omega}\log (\Lambda+\epsilon)\right)e^{(n-1)\phi}\omega^n=\int_X\left(\Delta_{e^\phi\omega}\log (\Lambda+\epsilon)\right)(e^{\phi}\omega)^n=0,
$$
where $\Delta_{e^\phi\omega}$ is the complex Laplacian with respect to $e^\phi\omega$. Then we plug it into \eqref{rigid_ineq_2} to see that
\begin{align}
\int_X\frac{\Lambda}{\Lambda+\epsilon}\left(\lambda_{\omega}-f^*\kappa_\eta\cdot \Lambda\right)e^{(n-1)\phi}\omega^n\le0\nonumber.
\end{align}
Moreover, we easily have a positive constant $L$ such that for any $\epsilon\in(0,1]$,
$$\left|\frac{\Lambda}{\Lambda+\epsilon}\left(\lambda_{\omega}-f^*\kappa_\eta\cdot \Lambda\right)\right|\le L$$
on $X$, and as $\epsilon\to0^+$,
$$\frac{\Lambda}{\Lambda+\epsilon}\left(\lambda_{\omega}-f^*\kappa_\eta\cdot \Lambda\right)\to\lambda_{\omega}-f^*\kappa_\eta\cdot \Lambda$$
pointwise on $X\setminus V$ and so pointwise almost everywhere on $X$ (note that $V$ is of zero measure with respect to $e^{(n-1)\phi}\omega^n$). Therefore, we can apply Lebegue's Dominated Convergence Theorem to conclude that
\begin{align}
\int_X\left(\lambda_{\omega}-f^*\kappa_\eta\cdot \Lambda\right)e^{(n-1)\phi}\omega^n&=\lim_{\epsilon\to0}\int_X\frac{\Lambda}{\Lambda+\epsilon}\left(\lambda_{\omega}-f^*\kappa_\eta\cdot \Lambda\right)e^{(n-1)\phi}\omega^n\nonumber\\
&\le0\nonumber,
\end{align}
from which Theorem \ref{main_const} follows.
\end{proof}

Next we prove Theorem \ref{main_de}. 

\begin{proof}[Proof of Theorem \ref{main_de}]
The proof uses ideas similar to Theorem \ref{main_const}. Set $u:=\frac{f^*\eta^n}{\omega^n}=\frac{(\det(\eta_{\alpha\bar\beta}\circ f))\cdot|\det J(f)|^2}{\det(g_{i\bar j})}$ (where we have fixed local holomorphic charts $(z^1,...,z^n)$ on $X$ and $(w^1,...,w^n)$ on $Y$ and write $\omega=\sqrt{-1}g_{i\bar j}dz^i\wedge d\bar z^j$ and $\eta=\sqrt{-1}\eta_{\alpha\bar \beta}dw^\alpha\wedge d\bar w^\beta$). Let $\epsilon$ be an arbitrary positive constant. At any point $x$ with $|\det J(f)|(x)>0$, i.e. $u(x)>0$, by computations in Chern \cite{Chern} and Lu \cite{L} we have
\begin{align}\label{de_ineq}
\Delta_{\omega}\log(u+\epsilon)&=\frac{\Delta_{\omega} u}{u+\epsilon}-\frac{|\partial u|^2}{(u+\epsilon)^2}\nonumber\\
&=\frac{u}{u+\epsilon}\left(\frac{\Delta_{\omega} u}{u}-\frac{|\partial u|^2}{u^2}\right)+\frac{\epsilon|\partial u|^2}{u(u+\epsilon)^2}\nonumber\\
&=\frac{u}{u+\epsilon}(\Delta_{\omega}\log u)+\frac{\epsilon|\partial u|^2}{u(u+\epsilon)^2}\nonumber\\
&=\frac{u}{u+\epsilon}tr_{\omega}\left(-f^*Ric(\eta)+Ric(\omega)\right)+\frac{\epsilon|\partial u|^2}{u(u+\epsilon)^2}\nonumber\\
&=\frac{u}{u+\epsilon}\left(-tr_{\omega}f^*Ric(\eta)+R_\omega\right)+\frac{\epsilon|\partial u|^2}{u(u+\epsilon)^2}\nonumber\\
&\ge \frac{u}{u+\epsilon}\left(-tr_{\omega}f^*Ric(\eta)+R_{\omega}\right),
\end{align}
where in the fourth equality we have used $\partial\bar\partial\log|\det J(f)|^2=0$ at $x$ whenever $\det J(f)(x)\neq0$. 
\par Set $W:=\{x\in X|\det J(f)=0$ at $x\}$, which, as $f$ is non-degenerate, is a proper subvariety (may be empty) of $X$. Then, the above discussions mean that \eqref{de_ineq} holds on $X\setminus W$ and by continuity we know it holds on the whole $X$. 
\par As before, we fix a $\psi\in C^\infty(X,\mathbb R)$ such that $e^\psi\omega$ is a Gauduchon metric on $X$. Therefore,
\begin{align}
\int_X\frac{u}{u+\epsilon}\left(-tr_{\omega}f^*Ric(\eta)+R_{\omega}\right)e^{(n-1)\psi}\omega^n&\le\int_{X}(\Delta_{\omega}\log(u+\epsilon))e^{(n-1)\psi}\omega^n\nonumber.
\end{align}
Now, we can use the same arguments as in Theorem \ref{main_const} to complete the proof.
\par Theorem \ref{main_de} is proved.
\end{proof}

\section{Applications}\label{app}
This section contains several applications of our integral inequalities.

\subsection{Rigidity theorems for holomorphic maps}
 
We will apply Theorem \ref{main_const} to obtain several rigidity theorems. For example, if we assume $(X,\omega)$ is a compact K\"ahler manifold of positive Ricci curvature (implying that $\int_X\lambda_{\omega}\omega^n>0$) and $(Y,\eta)$ is a K\"ahler manifold of nonpositive holomorphic sectional curvature (implying that $\kappa_\eta\le0$ on $Y$), then we easily recover the aforementioned fundamental rigidity theorems of Yau \cite{Y78} and Royden \cite{R} from Theorem \ref{main_const}. 
\par We may particularly mention that, in general, the function $\lambda_{\omega}$ in Theorem \ref{main_const} may has different signs at different points and so the second Ricci curvature may not be nonnegatively signed. Therefore, Theorem \ref{main_const} seems very flexible in applications, and implies new rigidity theorems even in the case that the curvatures of both the domain and target manifolds are not assume to be signed in pointwise sense. Theorem \ref{main_const} and the following applications indicate that the total first eigenvalue of second Ricci curvature should be essential in deriving rigidity theorems for holomorphic maps.

\begin{cor}\label{cor_1.1}
A holomorphic map from a compact Hermitian manifold of quasi-positive second Ricci curvature to a compact K\"ahler manifold of almost nonpositive holomorphic sectional curvature must be constant.
\end{cor}
Here, $\omega$ has quasi-positive second Ricci curvature if and only if $\lambda_\omega\ge0$ on $X$ and $\lambda_\omega>0$ at some point in $X$. Corollary \ref{cor_1.1} applies when then domain space is, e.g. a projective manifold of nef and big anti-canonical line bundle. 

\begin{proof}
Assume a contradiction that there is a non-constant holomorphic map $f:X\to Y$ between two compact complex manifolds, where $X$ admits a Hermitian metric $\omega$ of quasi-positive second Ricci curvature and $X$ is a compact K\"ahler manifold of almost nonpositive holomorphic sectional curvature . Fix $\phi\in C^\infty(X,\mathbb R)$ such that $e^\phi\omega$ is Gauduchon. Since $Ric^{(2)}(\omega)$ is quasi-positive, we know
\begin{equation}
\int_X\lambda_\omega e^{(n-1)\phi}\omega^n=:\delta_0>0\nonumber.
\end{equation}
Since $Y$ is a compact K\"ahler manifold of almost nonpositive holomorphic sectional curvature, we fix a sequence of K\"ahler metrics $\eta_i$, $i=0,1,2,...,$ on $Y$ such that $(\sup_XH^{\eta_i})\cdot[\eta_i]< i^{-1}[\eta_0]$. We may assume $\sup_XH^{\eta_i}>0$ for every $i$. Then $\kappa_{\eta_i}\le\sup_XH^{\eta_i}$ on $Y$, and so
\begin{align}
n\int_Xf^*\kappa_{\eta_i} \cdot e^{(n-1)\phi}f^*\eta_i\wedge\omega^{n-1}\le n\int_X (\sup_XH^{\eta_i})\cdot e^{(n-1)\phi}f^*\eta_i\wedge\omega^{n-1}\nonumber.
\end{align}
Recall that $e^\phi\omega$ being Gauduchon implies that the integration on the above right hand side depends only on the class of $\eta_i$. Therefore,
\begin{align}
n\int_Xf^*\kappa_{\eta_i} \cdot e^{(n-1)\phi}f^*\eta_i\wedge\omega^{n-1}&\le n\int_X (\sup_XH^{\eta_i})\cdot e^{(n-1)\phi}f^*[\eta_i]\wedge\omega^{n-1}\nonumber\\
&\le \frac ni\int_X e^{(n-1)\phi}f^*[\eta_0]\wedge\omega^{n-1}\nonumber\\
&\le\frac{\delta_0}{2}\nonumber\\
&<\int_X\lambda_\omega e^{(n-1)\phi}\omega^n\nonumber
\end{align}
for sufficiently large $i$, which contradicts Theorem \ref{main_const}.
\par Corollary \ref{cor_1.1} is proved.
\end{proof}

The same arguments prove the followings.

\begin{cor}\label{cor_1.2}
A holomorphic map from a compact Gauduchon manifold with positive total first eigenvalue of second Ricci curvature to a compact K\"ahler manifold of almost nonpositive holomorphic sectional curvature must be constant.
\end{cor}

\begin{cor}\label{cor_1.3}
A holomorphic map from a compact Gauduchon manifold with zero total first eigenvalue of second Ricci curvature to a K\"ahler (resp. Hermitian) manifold of negative holomorphic sectional (resp. bisectional) curvature must be constant.
\end{cor}

\begin{rem}
(1) In Corollary \ref{cor_1.2} we do not assume any pointwise curvature signs for both the domain and target spaces.\\
(2) In Corollaries \ref{cor_1.1} and \ref{cor_1.2}, the same conclusions hold if the target manifold is assumed to be a (not necessarily compact) Hermition manifold of nonpositive holomorphic bisectional curvature (or real bisectional curvature, a new curvature notion recently introduced in \cite{YZ}).\\
(3) It seems that Corollaries \ref{cor_1.1} and \ref{cor_1.2} are new even if we assume the target space is a K\"ahler (resp. Hermitian) manifold of nonpositive holomorphic sectional (resp. bisectional) curvature.\\
(4) Corollary \ref{cor_1.2} is somehow optimal if we further assume the target compact K\"ahler manifold is of nonpositive holomorphic sectional curvature. More precisely, for any compact K\"ahler manifold $Y$ of nonpositive holomorphic sectional curvature, there exist a compact K\"ahler manifold $X$ of zero total first eigenvalue of Ricci curvature and a non-constant holomorphic map $f:X\to Y$. In fact, since by \cite{Y78,TY} (also see \cite{Zys18} for a generalization to almost setting), a compact K\"ahler manifold of (almost) nonpositive holomorphic sectional curvature has nef canonical line bundle, then by using the nef reduction we have a fibration $\pi: Y\dashrightarrow Z$ with regular fiber $Y_z$, up to a finite unramified covering, a flat complex torus (see \cite{HLW16} for details). Then we may choose $X=Y_z$ and $f:X\to Y$ the holomorphic embedding.\\
(5) A very special case of Corollary \ref{cor_1.2} is that \emph{a K\"ahler manifold of almost nonpositive holomorphic sectional curvature contains no rational curves}, which has been observed in our previous work \cite[Theorem 1.10]{Zys18} by using a result of Tosatti-Y.G. Zhang \cite{ToZyg} (also see subsection \ref{subsect_cp1}). Our argument for Corollary \ref{cor_1.2} here provides an alternative proof for it.
\end{rem}

\subsection{Degeneracy theorems for holomorphic maps}
As another natural generalization of Schwarz-Pick-Ahlfors Lemma to higher dimensions, one may also compare the volume forms related by a holomorphic map. The most classical works include Chern \cite{Chern}, Lu \cite{L} and Yau \cite{Y78} etc.. In particular, Yau proved that a non-degenerate holomorphic map from a compact K\"ahler manifold of scalar curvature bounded from below by a negative number to a Hermitian manifold of Chern Ricci curvature bounded from above by a negative constant  decreases volume forms (up to multiplying a constant factor depending only on the bounds of curvatures), assuming the domain and target manifolds are of the same dimension. Consequently, a holomorphic map from an $n$-dimensional compact K\"ahler manifold of nonnegative scalar curvature to an $n$-dimensional Hermitian manifold of Chern Ricci curvature bounded from above by a negative constant must be totally degenerate. More recent developments can be found in \cite{To,Ni} etc.. 
\par Here we would like to prove some degeneracy theorems using the integral inequality in Theorem \ref{main_de}. Given Theorem \ref{main_de}, to get degeneracy theorems, we just need to assume curvature conditions that will destroy \eqref{ineq_de}. Again, noting that it is an integral inequality, the curvatures may be assumed to be signed only in certain integral or almost sense, not necessarily in pointwise sense.
\begin{cor}\label{cor_2.2}
A holomorphic map from an $n$-dimensional compact Gauduchon manifold of zero total Chern scalar curvature to an $n$-dimensional compact K\"ahler manifold of semiample canonical line bundle and positive Kodaira dimension must be totally degenerate.
\end{cor}
\begin{proof}
Assume a contradiction that there is a non-degenerate holomorphic map $f:X\to Y$ between two $n$-dimensional compact complex manifolds, where $X$ admits a Gauduchon metric $\omega$ of zero total Chern scalar curvature, i.e.
\begin{equation}
\int_XR_\omega \omega^n=0\nonumber,
\end{equation} 
and $Y$ is a compact K\"ahler manifold of semiample canonical line bundle and positive Kodaira dimension. Since $f$ is non-degenerate at some point $x_0\in X$, $f|_U:U\to f(U)$ is biholomorphic for some open neighborhood $U$ of $x_0$. Let $\pi:Y\to \pi(Y)\subset \mathbb{CP}^N$ be the semiample fibration induced by the pluricanonical linear system of $Y$, where the Kodaira dimension of $Y$ equals to $dim_{\mathbb C}(\pi(Y))\ge1$. Then we fix $\chi=\pi^*\tilde\omega\in2\pi c_1(K_X)$, where $\tilde\omega$ is a multiple of Fubini-Study metric on $\mathbb{CP}^N$, and by Yau's fundamental theorem \cite{Y} we fix a K\"ahler metric $\eta$ on $Y$ such that $Ric(\eta)=-\chi$. Since $\chi$ is semipositive on $Y$ and, at a generic point, $\chi$ is positive in some directions, up to shrinking $U$ (and so $f(U)$) we may assume $\chi$ has positive directions for every point in $f(U)$. Therefore, 
\begin{align}
n\int_Xf^*(Ric(\eta))\wedge\omega^{n-1}&=-n\int_Xf^*\chi\wedge\omega^{n-1}\nonumber\\
&\le-n\int_Uf^*\chi\wedge\omega^{n-1}\nonumber\\
&<0=\int_XR_\omega \omega^n\nonumber,
\end{align}
contradicting \eqref{ineq_de}.
\par Corollary \ref{cor_2.2} is proved.
\end{proof}

Similar arguments can be used to prove
\begin{cor}\label{cor_2.3}
A holomorphic map from an $n$-dimensional compact Gauduchon manifold of positive total Chern scalar curvature to an $n$-dimensional compact K\"ahler manifold of nef canonical line bundle must be totally degenerate.
\end{cor}

\begin{rem}
We should point out that Corollary \ref{cor_2.3} is essentially not new. In fact, by a recent work of Yang \cite[Theorem 4.1]{Y19}, a compact complex manifold admitting a Gauduchon metric of positive total Chern scalar curvature also admits a Hermitian metric of (pointwise) positive Chern scalar curvature. Therefore, using Yang's result, Corollary \ref{cor_2.3} follows directly from the classical maximum principle arguments in \cite{Chern,L,Y}. 
\end{rem}

Here, using the integral inequality \eqref{ineq_de}, we shall provide an alternative proof for Corollary \ref{cor_2.3} without involving Yang's result.

\begin{proof}
Assume a contradiction that there is a non-degenerate holomorphic map $f:X\to Y$ between two $n$-dimensional compact complex manifolds, where $X$ admits a Gauduchon metric $\omega$ of positive total Chern scalar curvature, i.e.
\begin{equation}
\int_XR_\omega \omega^n=:\delta_0>0\nonumber,
\end{equation}
and $Y$ is a compact K\"ahler manifold of nef canonical line bundle. 
By the definition of nefness and Yau's fundamental theorem \cite{Y} we may fix a sequence of K\"ahler metrics $\eta_i$, $i=0,1,2,...,$ on $Y$ such that 
$$Ric(\eta_i)\le i^{-1}\eta_0.$$
Then we easily see that 
\begin{align}
n\int_Xf^*(Ric(\eta_i))\wedge\omega^{n-1}&\le\frac ni\int_Xf^*(\eta_0)\wedge\omega^{n-1}\nonumber\\
&\le\frac{\delta_0}{2}\nonumber\\
&<\int_XR_\omega \omega^n\nonumber
\end{align}
for sufficiently large $i$, contradicting \eqref{ineq_de}.
\par Corollary \ref{cor_2.3} is proved.
\end{proof}

\subsection{Infinite-time singularity types of the K\"ahler-Ricci flow} We would like to discuss how our integral inequality \eqref{main.ineq} relates to the study of infinite-time singularity types of the K\"ahler-Ricci flow. Let $(Y,\eta_0)$ be an $m$-dimensional compact K\"ahler manifold. Consider the K\"ahler-Ricci flow $\eta=\eta(t)$, $t\in[0,T)$ on $Y$ running from $\eta_0$:
\begin{equation}\label{krf}
\left\{
\begin{aligned}
\partial_t\eta(t)&=-Ric(\eta(t))\\
\eta(0)&=\eta_0.
\end{aligned}
\right.
\end{equation}
We assume the canonical line bundle $K_Y$ of $Y$ is nef, which is equivalent to that the K\"ahler-Ricci flow running from an arbitrary K\"ahler metric can be smoothly solved for all $t\in[0,\infty)$ (see \cite{C,TZo,Ts}). Recall from Hamilton \cite{Ha93} that the infinite-time singularities of the K\"ahler-Ricci flow are divided into three types. Precisely, a long-time solution $\eta(t)$, $t\in[0,\infty)$, to the K\"ahler-Ricci flow \eqref{krf} is of
\begin{itemize}
\item type \textrm{IIb} if
\begin{equation}\label{defn_ii}
\limsup_{t\to\infty}\left(\sup_{Y}t|Rm(\eta(t))|_{\eta(t)}\right)=\infty,\nonumber
\end{equation}
\item  type  \textrm{IIIa} if
\begin{equation}\label{defn_iii}
\limsup_{t\to\infty}\left(\sup_{Y}t|Rm(\eta(t))|_{\eta(t)}\right)\in(0,\infty)\nonumber,
\end{equation}
\item type  \textrm{IIIb} if
\begin{equation}\label{defn_iii}
\limsup_{t\to\infty}\left(\sup_{Y}t|Rm(\eta(t))|_{\eta(t)}\right)=0\nonumber,
\end{equation}
\end{itemize}
The infinite-time singularity types are about the long-time boundedness of curvature tensor of the K\"ahler-Ricci flow, which are crucial in understanding the singularity models of the K\"ahler-Ricci flow. There are many progresses in classifying infinite-time singularity types of the K\"ahler-Ricci flow in recent years, assuming the Abundance Conjecture (i.e. the canonical line bundle is semiample), see   \cite{FZo,HeTo,ToZyg,Zys17,Zys19,FZys}. In the surface case, a complete classification is obtained by Tosatti-Y.G.Zhang \cite{ToZyg}. In general, without assuming the Abundance Conjecture, it seems not much progresses in determining the singularity type of the K\"ahler-Ricci flow. A conjecture raised by Tosatti \cite[Conjecture 6.7]{To.kawa} predicts that for the K\"ahler-Ricci flow on any compact K\"ahler manifold with nef canonical line bundle, the infinite-time singularity type does not depend on the choice of the initial metric, which is confirmed in our previous work \cite[Corolloar 1.5]{Zys17} in $3$-dimensional case.

\par Here we would like to first recall a useful criterion for type \textrm{II}b singularities due to Tosatti-Y.G.Zhang \cite{ToZyg} without assuming the Abundance Conjecture. 

\begin{prop}\label{tz}\cite[Proposition 1.4]{ToZyg}
Let $Y$ be a compact K\"ahler manifold of nef canonical line bundle. Assume there is a non-constant holomorphic map $f:\mathbb{CP}^1\to Y$ such that $\int_{\mathbb{CP}^1}f^*c_1(Y)=0$, then any solution to the K\"ahler-Ricci flow \eqref{krf} on $Y$ must be of type \textrm{II}b.
\end{prop}

In their proof \cite[Section 4, Remark 4.1]{ToZyg} for Proposition \ref{tz}, an integral inequality for a non-constant holomorphic map $f:\mathbb{CP}^1\to Y$ is derived using ``$\epsilon$-regularity argument". Our Theorem \ref{main_const} provides an alternative argument for \cite[Section 4, Remark 4.1]{ToZyg}, and generalizes it to a more general setting. Consequently, we can also generalize the criterion in Proposition \ref{tz} as follows.

\begin{cor}\label{cor_st}
Let $Y$ be a compact K\"ahler manifold of nef canonical line bundle. Assume there is a non-constant holomorphic map $f:(X,\omega)\to Y$, where $(X,\omega)$ is a compact Gauduchon manifold of positive total second Ricci curvature, then any solution to the K\"ahler-Ricci flow \eqref{krf} on $Y$ must be of type \textrm{II}b or type \textrm{III}a. If furthermore $\int_Xf^*c_1(Y)\wedge\omega^{n-1}=0$ (where $n=dim_{\mathbb C}X$), then any solution to the K\"ahler-Ricci flow \eqref{krf} on $Y$ must be of type \textrm{II}b.
\end{cor}
\begin{proof}
The proof uses the ideas in Corollaries \ref{cor_1.1} and \ref{cor_1.2}. Given a long-time solution $\eta(t)$ to the K\"ahler-Ricci flow \eqref{krf} on $Y$. Similar to Corollaries \ref{cor_1.1} and \ref{cor_1.2}, we have
\begin{align}
\sup_YH^{\eta(t)}&\ge\frac{\int_X\lambda_{\omega}\omega^n}{n\int_Xf^*\eta(t)\wedge\omega^{n-1}}\nonumber\\
&=\frac{n^{-1}\int_X\lambda_{\omega}\omega^n}{\int_Xf^*\eta_0\wedge\omega^{n-1}+t\cdot2\pi \int_Xf^*c_1(K_Y)\wedge\omega^{n-1}}\nonumber,
\end{align}
where we have used that $[\eta(t)]=[\eta_0]+t\cdot2\pi c_1(K_Y)$ and $\omega$ is Gauduchon. Therefore,
as $t\to\infty$, 
\begin{itemize}
\item $\limsup_{t\to\infty}\sup_YtH^{\eta(t)}\ge\frac{n^{-1}\int_X\lambda_{\omega}\omega^n}{2\pi \int_Xf^*c_1(K_Y)\wedge\omega^{n-1}}>0$, if $\int_Xf^*c_1(K_Y)\wedge\omega^{n-1}>0$;
\item $\limsup_{t\to\infty}\sup_YtH^{\eta(t)}=+\infty$, if $\int_Xf^*c_1(K_Y)\wedge\omega^{n-1}=0$.
\end{itemize}

Corollary \ref{cor_st} is proved.
\end{proof}

\begin{rem}
In Corollary \ref{cor_st}, if $Y$ is a compact complex manifold and we are given a long-time solution to the \emph{Chern-Ricci flow} (see \cite[Equation (1.1)]{TW}), say $\eta(t)$, on $Y$, then similar curvature estimates hold if we replace $H^{\eta(t)}$ by $BK^{\eta(t)}$ and $c_1(Y)$ by the first Bott-Chern class $c_1^{BC}(Y)$.
\end{rem}

\subsection{Lower bounds for $\mu_\alpha$ and nef threshold}\label{final_rem}
In this subsection, we mention that the integral inequalities \eqref{main.ineq} and \eqref{ineq_de} imply effective estimates for the number $\mu_\alpha$ defined in Definition \ref{defn1} and the nef threshold $\nu_\alpha$ of $\alpha$, where $\alpha$ is a K\"ahler class on a compact K\"ahler manifold $Y$ and $\nu_\alpha:=\inf\{s\in\mathbb R|2\pi c_1(K_Y)+s\alpha$ is nef$\}$. Namely, 
\begin{itemize}
\item given any non-constant holomorphic map $f:(X,\omega)\to Y$ with $(X,\omega)$ a compact Gauduchon manifold, we have
$$\mu_\alpha\ge\frac{\int_X\lambda_{\omega}\omega^n}{n\int_Xf^*\alpha\wedge\omega^{n-1}}.$$
\item given any non-degenerate holomorphic map $f:(X,\omega)\to Y$ with $(X,\omega)$ a compact Gauduchon manifold and $dim_{\mathbb C}X=dim_{\mathbb C}Y$, we have
$$\nu_\alpha\ge\frac{\int_XR_{\omega}\omega^n}{n\int_Xf^*\alpha\wedge\omega^{n-1}}.$$
\end{itemize}
Therefore, we have obtained the desired effectiveness results discussed in Subsection \ref{moti}, which motivate our study, in a much more general setting. These may also be regarded as quantitative versions of Corollaries \ref{cor_1.2} and \ref{cor_2.3}, respectively.

\subsection{Characterizing the equality case} For integral inequality \eqref{main.ineq} in Theorem \ref{main_const}, we may wonder: can we conclude any particular properties if the equality is achieved? While that inequality always holds, we expect that the equality case would give restrictions on the involved data. To be more precise, let's focus on a special setting as follows.
\par In Theorem \ref{main_const}, if we choose $X=Y$ be an $n$-dimensional compact K\"ahler manifold and $f=Id_X:X\to X$ the identity map, then for any K\"ahler metrics $\omega$ and $\eta$ on $X$ we have
\begin{equation}\label{ineq_ut_0}
\int_X\lambda_\omega\omega^n\le n\int_X\kappa_\eta\cdot\eta\wedge\omega^{n-1}.
\end{equation}

Then it may be natural to ask:
\begin{ques}
Can we conclude any particular properties if the equality in \eqref{ineq_ut_0} is achieved by two K\"ahler metrics $\omega$ and $\eta$ on $X$?
\end{ques}

It is possible to answer the above question for some special settings.

\begin{prop}[$\eta$ being a K\"ahler-Einstein metric]
Let $(X,\eta)$ be an $n$-dimensional compact K\"ahler-Einstein manifold. For any given K\"ahler metric $\omega$ on $X$ we have
\begin{equation}\label{ineq_ut_00}
\int_X\lambda_\omega\omega^n\le n\int_X\kappa_\eta\cdot\eta\wedge\omega^{n-1},
\end{equation}
and the equality holds for some $\omega$ if and only if both $\eta$ and $\omega$ are of zero or negative constant holomorphic sectional curvature.
\end{prop}
\begin{proof}
Firstly, recall a classical result of Berger: for a K\"ahler metric $\eta$ on $X$ and any point $x\in X$, we have
\begin{align}
\int_{W\in\mathbb{CP}^{n-1}}H^\eta_x(W)\omega_{FS}^{n-1}(W)=\frac{2}{n(n+1)}R_\eta(x),
\end{align}
where $\omega_{FS}$ is the Fubini-Study metric on $\mathbb{CP}^{n-1}$ with $\int_{\mathbb{CP}^{n-1}}\omega_{FS}^{n-1}=1$ and $H^\eta_x$ is regarded as a function defined on $\mathbb{CP}^{n-1}$. Combining the definition of $\kappa_\eta$ in Notation \ref{hc}(2) implies
\begin{align}\label{ineq_ut_01}
n\kappa_\eta\cdot\eta\ge\frac{n+1}{2}H^\eta\cdot\eta\ge\frac{1}{n}R_\eta\eta=Ric(\eta),
\end{align}
where the first inequality is strict at those points $x$ with $H^\eta_x>0$ and in the last equality we used $\eta$ is K\"ahler-Einstein.
Now, if there is a K\"ahler metric $\omega$ satisfying the equality in \eqref{ineq_ut_00}, we see
$$\int_X\lambda_\omega\omega^n=n\int_X\kappa_\eta\cdot\eta\wedge\omega^{n-1}\ge\int_XRic(\eta)\wedge\omega^{n-1}=\int_XRic(\omega)\wedge\omega^{n-1},$$
implying that
\begin{align}\label{eq_ke}
Ric(\omega)=\lambda_\omega\omega
\end{align}
and hence
$$n\kappa_\eta\cdot\eta=Ric(\eta).$$
It turns out that the inequalities in \eqref{ineq_ut_01} are all equalities. Therefore, the holomorphic sectional curvature of $\eta$ is nonpositive and \emph{pointwise} constant, equaling to $\frac{2}{n(n+1)}R_\eta$, and hence is a constant, since $\eta$ is K\"ahler-Einstein. 
\par On the other hand, the \eqref{eq_ke} implies that $\lambda_\omega$ is smooth, and then by differentiating \eqref{eq_ke} we see $d \lambda_\omega\equiv0$, i.e. $\lambda_\omega$ is a constant and $\omega$ is a K\"ahler-Einstien metric on $X$. Therefore, $\omega$ is also of constant holomorphic sectional curvature, since $(X,\eta)$ is a compact complex space form of zero or negative curvature. Note that in the negative curvature case, by uniqueness of negative K\"ahler-Einstein metric, $\omega$ must be proportional to $\eta$.
\end{proof}

By almost identical arguments, we also have

\begin{prop}[$\eta$ being a cscK metric]
Let $(X,\eta)$ be an $n$-dimensional compact K\"ahler manifold of constant scalar curvature. For any given K\"ahler metric $\omega$ on $X$ we have
\begin{equation}\label{ineq_ut_000}
\int_X\lambda_\omega\omega^n\le n\int_X\kappa_\eta\cdot\eta\wedge\omega^{n-1},
\end{equation}
and the equality holds for some $\omega\in[\eta]$ if and only if $\eta$ is of zero or negative constant holomorphic sectional curvature and $\omega=\eta$.
\end{prop}

\section*{Acknowledgements}
The author is grateful to professor Huai-Dong Cao for valuable discussions and suggestions during the study of related topics, and constant encouragement and support. He also thanks professors Valentino Tosatti for comments on our results, Jian Xiao for a number of discussions, Xiaokui Yang for pointing out results in \cite{Y19}, Hui-Chun Zhang and Fangyang Zheng for interest and comments on a previous version and the referees for valuable comments and suggestions.

\end{document}